\documentclass[leqno,twoside, 11pt]{amsart}
\usepackage{amssymb, latexsym, esint}
\usepackage{color}

\usepackage{amsthm}
\usepackage{amssymb,amsfonts}
\usepackage[all,arc]{xy}
\usepackage{enumerate}
\usepackage{mathrsfs}
\usepackage{amsmath}
\usepackage{verbatim}

\theoremstyle{plain}
\numberwithin{equation}{section} \numberwithin{figure}{section}

\newtheorem{theorem}{Theorem}[section]
\newtheorem{lemma}[theorem]{Lemma}
\newtheorem{proposition}[theorem]{Proposition}
\newtheorem{corollary}[theorem]{Corollary}
\newtheorem{definition}[theorem]{Definition}
  \usepackage{epsfig} 
\usepackage{graphicx}  \usepackage{epstopdf}
\theoremstyle{definition}
\newtheorem{remark}[theorem]{Remark}

\numberwithin{equation}{section}

\begin{document}
\title[Fractional elliptic equations with first-order terms ]{A capacity-based  condition for existence of solutions to fractional elliptic equations with first-order terms and measures} 
\author{Mar\'ia Laura de Borb\'on  and Pablo Ochoa}
\address{Maria Laura de Borb\'on, Universidad Nacional de Cuyo-CONICET, Mendoza 5500, Argentina} 
\email{laudebor@gmail.com}
\address{Pablo Ochoa, Universidad Nacional de Cuyo-CONICET, Mendoza 5500, Argentina} 
\email{ochopablo@gmail.com}




\date{\today}
\begin{abstract} 
In this manuscript, we appeal to Potential Theory to provide a sufficient condition for existence of distributional solutions to fractional elliptic problems with non-linear first-order terms and measure data $\omega$:
$$
\left\{
\begin{array}{rcll}
  (-\Delta)^su&=&|\nabla u|^q + \omega \quad \text{in }\mathbb{R}^n,\, \,\,s \in (1/2, 1)\\
  u & > &0 \quad  \text{in } \mathbb{R}^{n}\\
  \lim_{|x|\to \infty}u(x) & =& 0,
\end{array}
\right.
$$under suitable assumptions on $q$ and $\omega$. Roughly speaking, the condition for exis\-tence states that if the measure data  is locally controlled by the Riesz fractional capacity, then there is a global solution for the equation.   
We also show that if  a positive solution exists, necessarily the measure $\omega$ will be absolutely continuous with respect to the associated Riesz capacity, which gives a partial reciprocal of the main result of this work. Finally, estimates of $u$ in terms of $\omega$ are also given in different function spaces.

\end{abstract}
\maketitle

\section{Introduction }

We study the solvability of the following fractional elliptic problem
\begin{equation}\label{problem}
	\left\{
	\begin{array}{rcll}
	(-\Delta)^su&=&|\nabla u|^q + \omega \quad \text{in }\mathbb{R}^n\\
	u & > &0 \quad  \text{in } \mathbb{R}^{n}\\
	\lim_{|x|\to \infty}u(x) & =& 0,
	\end{array}
	\right.
	\end{equation}where $\frac{1}{2}<s<1$, $n>2s$ and $(-\Delta)^s$ is the classic  fractional Laplacian operator of order $2s$. Here $\omega$ will be a non-negative Radon measure with compact support in $\mathbb{R}^n$.  We consider the super-critical case 
$$q>p^*=\dfrac{n}{n-2s+1},\quad p+q=pq.$$This assumption on $q$ is motivated from the fact that  in the local-case, as we shall detail below, no solutions exist for sub-critical $q$ unless $\omega \equiv 0$. For $W^{1, q}(\mathbb{R}^{n})$-solutions, a similar conclusion is obtained in our framework (we  refer the reader  to Remark \ref{rmk converse} for details).   Moreover,  the super-critical case  allows us to  obtain basic estimates on potentials  as will be clarify  in the proof of the main results.

We  highlight that non-local type operators arise naturally  in continuum mechanics, image processing, crystal dislocation, Non-linear Dynamics (Geophysical Flows), phase transition phenomena, population dynamics, non-local optimal control and game theory  (\cite{BCF}, \cite{BCF12}, \cite{BV}, \cite{C}, \cite{CV10}, \cite{CV}, \cite{DFV}, \cite{DPV}, \cite{GO}, \cite{L00} and the references therein). Indeed, models like  \eqref{eq} may be understood as a Kardar-Parisi-Zhang stationary problem (models of growing interfaces) driving by fractional diffusion (see \cite{KPZ} for the model in the local setting and \cite{AyP} in the nonlocal stage).  In the works \cite{MK} and \cite{MK1} the description of anomalous diffusion via fractional dynamics is investigated and various fractional partial differential equations are derived from L\'evy random walk models, extending Brownian motion models in a natural way. Finally, fractional type operators are also encompassed in mathematical modeling of financial markets, since L\'{e}vy type processes with jumps take place as more accurate models of stock pricing (cf. \cite{A04} and \cite{CT} for some illustrative examples).

In this work we provide a sufficient condition for existence of global solutions to \eqref{eq} based in a relation between $\omega$ and a fractional capacity. We also derive a representation formula for the solution $u$ in terms of Riesz potentials and, as a result, we obtain pointwise and norm estimates and behaviour at infinity of $u$. Finally we demonstrate a necessary condition for the existence of solutions to the problem \eqref{problem}. We show that if problem \eqref{problem} has a positive $W^{1,q}(\mathbb{R}^n)$-solution, then  $\omega$ does not charge sets of Riesz capacity zero, which means $\omega$ will necessarily be absolutely continuous with respect to the corresponding Riesz capacity. This result gives us a partial reciprocal of the main theorem.  We refer the reader to the next section for further details. 

The current  approach of the problem has been inspired by the enlightening results from \cite{Han}. In that paper, the authors provide a criteria for existence of solutions to equations of the form
\begin{equation}\label{local}
-\Delta u = |\nabla u|^{q} + \omega 
\end{equation}in $\mathbb{R}^{n}$. They also prove the same characterization in bounded domains, but in this case $q\geq 2$ is needed for sufficient conditions.  The criteria of solvability is given explicitly in terms of pointwise behaviour of the corresponding Riesz potentials as well as in geometric capacitary terms. We also note that in a parallel theory of equations of the form
$$-\Delta u=|u|^{q}+ \omega$$the role played by the Riesz capacity of order $(2, q')$ is analogous to the results in \cite{Han} (see for instance \cite{AP} and \cite{KV}). As a consequence of the results in \cite{Han} and the potential estimates from \cite{MV} (see also \cite[Section 7.2]{AH} and \cite[Section 11.5]{Ma}) it follows that no solution exists to \eqref{local} when $1 < q \leq n/(n-1)$ unless $\omega \equiv 0$. 

A different approach to characterize existence of solutions to local elliptic equations  in terms of capacities was stated in \cite{BGO}. There, it was proved that a measure is absolutely continuous with respect to $(1, p)$ Riesz capacity if and only if it belongs to $L^{1}(\Omega)+ W^{-1, p'}(\Omega)$. This characterization allows to prove that a solution exists if and only if the measure data is absolutely continuous with respect to the Riesz capacity. Extensions of these results may be found, for instance,  in \cite{BGO2}.

We point out that the problem of existence  of solutions to fractional elliptic problems with first-order terms like \eqref{eq} in bounded domains (with boundary data) has been considered in \cite{CV1}, \cite{CV} and \cite{AyP}. We highlight that our approach is also comparable to \cite[Section 5]{AyP} where a sufficient condition in terms of fractional capacity is also obtained for bounded domains and highly integrable sources. For  fractional diffusion  problems with non-homogeneous boundary conditions we refer to \cite{CH}.

The paper is organized as follows. In Section \ref{preliminaries}, we provide the basic notation and definitions and also state the main result of the paper.  Consequences of the main result, such as representation formula for the solution, pointwise and norm estimates and behaviour at infinite, are also provided.  In Section \ref{proof} we give full details on the proof of Theorem \ref{Theo1}. Finally, Section \ref{converse sec} is dedicated to show the necessary condition for existence of solutions to problem \eqref{problem} and also discuss the converse of the main theorem which is currently an open problem.  

\section{Preliminaries and main results }\label{preliminaries}
The Fourier definition of the fractional Laplace operator $(-\Delta)^s$ is given by



\begin{equation}\label{fracLaplacianFou}
(-\Delta)^su(x)=\mathcal{F}^{-1}(|\xi|^{2s}\mathcal{F}(u)(\xi))(x)
\end{equation}$\xi\in\mathbb{R}^n$ and $u\in \mathcal{S}(\mathbb{R}^n)$, where  $\mathcal{S}(\mathbb{R}^n)$ is the Schwartz class of smooth real-valued rapidly decreasing
functions.  We recall the following integral formulation of $(-\Delta)^{s}$ for functions in $\mathcal{S}(\mathbb{R}^{n})$
$$(-\Delta)^{s}u(x) :=a(n, s) \text{P. V.}\int_{\mathbb{R}^{n}}\frac{u(x)-u(y)}{|x-y|^{n+2s}}dy$$where $$a(n, s)=\frac{2^{2s}s\Gamma\left(\frac{n}{2}+s \right) }{\pi^{\frac{n}{2}}\Gamma(1-s)}$$is a normalization constant to recover \eqref{fracLaplacianFou}. We refer to \cite{Land} and \cite{S} for extensions to H\"{o}lder function spaces.

Following the approach in \cite{S}, we consider the following spaces

$$\mathcal{S}_s(\mathbb{R}^n):=\left\lbrace f\in \mathcal{C}^\infty(\mathbb{R}^n): \forall\beta\in\mathbb{N}_0^n, \sup_{x\in\mathbb{R}^n}(1+|x|^{n+2s})|D^\beta f(x)|<\infty\right\rbrace$$with the family of seminorms

$$[f]_{\mathcal{S}_s(\mathbb{R}^n)}^\beta := \sup_{x\in\mathbb{R}^n}(1+|x|^{n+2s})|D^\beta f(x)|$$and the weighted  Lebesgue space

$$L_s(\mathbb{R}^n):=\left\lbrace u\in L^1_{\textnormal{loc}}(\mathbb{R}^n) : \int_{\mathbb{R}^n}\frac{|u(x)|}{1+|x|^{n+2s}}\;dx<\infty \right\rbrace.$$ We denote by $\mathcal{S}_s'$ the topological dual of $\mathcal{S}_s$.
It is easy to see that $L_s(\mathbb{R}^n)\subset \mathcal{S}_s'(\mathbb{R}^n)$ and if $\varphi\in\mathcal{S}(\mathbb{R}^n)$, then $(-\Delta)^s\varphi\in \mathcal{S}_s(\mathbb{R}^n)$ (see \cite{Buc}). 
Then, for $u\in L_s(\mathbb{R}^n)$, we can define $(-\Delta)^su$ in sense of tempered distributions as

$$\left\langle (-\Delta)^su, \varphi\right\rangle= \left\langle u,(-\Delta)^s\varphi\right\rangle= \int_{\mathbb{R}^n}u(x)(-\Delta)^s\varphi(x)\;dx$$for $\varphi\in \mathcal{S}(\mathbb{R}^n)$. (See \cite{Mar}).

We shall consider weak solution of \eqref{eq} in the following sense.  We start with the case without first-order terms. 
\begin{definition}\label{Defsolomega}
	Consider the equation 
		\begin{equation}\label{eqomega}
		(-\Delta)^s u = \omega.
		\end{equation}Then, $u\in L_s(\mathbb{R}^n)$ is a weak (distributional) solution of \eqref{eqomega} if 
		\begin{equation*}
		\int_{\mathbb{R}^n}u(x)(-\Delta)^s\varphi(x)\,dx = \int_{\mathbb{R}^n}\varphi(x)\,d\omega(x)
		\end{equation*}for all $\varphi\in \mathcal{S}(\mathbb{R}^n)$. If, instead of $\omega$ we have an integrable function $f$, we replace $d\omega(x)$ by $f(x)dx$.		

\end{definition}
\begin{definition}\label{Defsol}
	
	We say that $u\in L_s(\mathbb{R}^n)\cap W_{\textnormal{loc}}^{1,q}(\mathbb{R}^n)$ is a weak solution of the equation \eqref{eq} if for all $\varphi\in S(\mathbb{R}^n)$, $|\nabla u|^{q}\varphi \in L^{1}_{\text{loc}}(\mathbb{R}^{n})$ and 
	\begin{equation*}
	\int_{\mathbb{R}^n}u(x)(-\Delta)^s\varphi(x)\;dx = \int_{\mathbb{R}^n}|\nabla u(x)|^q\varphi(x)\;dx + \int_{\mathbb{R}^n}\varphi(y)\;d\omega(y).
	\end{equation*}
	
\end{definition}

Next we define the Riesz potential $\mathcal{I}_{\alpha}=(-\Delta)^{-\frac{\alpha}{2}}$ on $\mathbb{R}^n$ of order $\alpha$ for $0<\alpha<n$ as follows

$$\mathcal{I}_{\alpha}(g)(x)=c(n,\alpha)\int_{\mathbb{R}^n}\frac{g(y)}{|x-y|^{n-\alpha}}\;dy$$for $g\in L^1_{\textnormal{loc}}(\mathbb{R}^n)$ such that $$\int_{|y|\geq 1}\frac{|g(y)|}{|y|^{n-\alpha}}\;dy<\infty.$$ The constant $c(n,\alpha)$ is defined by

$$c(n,\alpha)=\pi^{-\frac{n}{2}}2^{-\alpha}\Gamma\left( \frac{n-\alpha}{2}\right) \Gamma\left( \frac{\alpha}{2}\right)^{-1}.$$For a non-negative Radon measure $\omega$, we define the Riesz potential $$\mathcal{I}_\alpha(\omega)(x)=c(n,\alpha)\int_{\mathbb{R}^n}\frac{1}{|x-y|^{n-\alpha}}\;d\omega(y).$$The kernel $$I_\alpha(x)=\frac{c(n,\alpha)}{|x|^{n-\alpha}}$$is called the Riesz kernel. In this way, we see that 

$$\mathcal{I}_\alpha (g)(x)=(I_\alpha*g)(x)$$and
$$\mathcal{I}_\alpha(\omega)(x)=(I_\alpha*\omega)(x).$$

 In order to state our main result, we first give the definition of Riesz capacity that will be employed along the work.

\begin{definition}
	
	For $0<\alpha<n$ and $1<q<\infty$, the Riesz capacity $\textnormal{cap}_{\alpha,q}(E)$ of a measurable set $E\subset\mathbb{R}^n$ is defined by
\begin{equation}\label{cap}
\textnormal{cap}_{\alpha,q}(E)=\inf\left\lbrace \|u\|_{L^q(\mathbb{R}^n)}^q:\mathcal{I}_\alpha(u)\geq \mathcal{X}_E,\;u\in L^q_+(\mathbb{R}^n)\right\rbrace. 
\end{equation}

\end{definition}

The main result of this work is the following theorem.

\begin{theorem}\label{Theo1}
	
	Let $\frac{1}{2}<s<1$, $q>p^*$. Suppose that there exists  $C=C(n, q, s)> 0$, such that
	\begin{equation}\label{cap1}
	\omega(E)\leq C\textnormal{cap}_{2s-1,q'}(E)
	\end{equation}for all compact sets $E\subset\mathbb{R}^n$, then the equation
	\begin{equation}\label{eq}
(-\Delta)^su=|\nabla u|^q + \omega \quad \text{in }\mathbb{R}^n
\end{equation} has a non-negative weak solution $u\in W_{\textnormal{loc}}^{1,q}(\mathbb{R}^n)$. Moreover, $u$ has the representation
		\begin{equation}\label{represent}
u = \mathcal{I}_{2s}(|\nabla u|^q+\omega)
\end{equation}

\end{theorem}

\

We point out  that \eqref{cap1} is equivalent,  by Theorem 2.1 from \cite{MV}, to
\begin{equation}\label{ineqpot}	
		\mathcal{I}_{2s-1}\left( \left[ \mathcal{I}_{2s-1}(\omega)\right] ^q\right)\leq C_1 \mathcal{I}_{2s-1}(\omega),  \quad \text{a. e. in }\mathbb{R}^{n}
		\end{equation}for some $C_1>0$ depending on $n$, $q$ and $s$. This relation, involving the non-lineal potential $\mathcal{I}_{2s-1}\left( \left[ \mathcal{I}_{2s-1}(\omega)\right] ^q\right)$, will be frequently used in the proof of Theorem \ref{Theo1}.

		The following result constitutes a partial converse of Theorem \ref{Theo1}. For the feasibility of  the $W^{1, q}(\mathbb{R}^{n})$-regularity of solutions, we refer the reader to Proposition \ref{Coro}. 
		\begin{theorem}\label{partial converse}
		Suppose that \eqref{eq} has a solution $u \in W^{1, q}(\mathbb{R}^{n})$. Then the measure $\omega$ does not charge sets of Riesz capacity zero.
		\end{theorem}\begin{remark}\label{rmk converse} In other words, Theorem \ref{partial converse} says that a necessary condition for existence of solutions is that $\omega$ is absolutely continuous with respect to the Riesz capacity. A direct consequence of this fact is that no global solution $u$ in $ W^{1, q}(\mathbb{R}^{n})$ exists if $q \leq  p^{*}$ (see \cite[Propostition 2.6.1]{AH}).\end{remark} 
		
		In the rest of this section, we shall provide some consequences of Theorem \ref{Theo1}.

		\begin{proposition}\label{Coro} The solution from Theorem \ref{Theo1} satisfies
\begin{equation}\label{est 1 u}
\mathcal{I}_{2s}(\omega) \leq u \leq C\mathcal{I}_{2s}(\omega)
\end{equation}		
		and
		\begin{equation}\label{est grad u}
		|\nabla u|\leq C\mathcal{I}_{2s-1}(\omega).
		\end{equation}As a result, if $q> n/(n-2s)$, we have $u \in W^{1, q}(\mathbb{R}^{n}).$
		\end{proposition}\begin{proof}
		 Observe that the estimate \eqref{est 1 u} is obtained as follows. From \eqref{estimate 2} we have
		\begin{equation}\label{estimate 23}
	u\leq C\left( \mathcal{I}_{2s}\omega + \mathcal{I}_{2s}\left[ \mathcal{I}_{2s-1}\omega\right]^q \right)
	\end{equation}Now, applying $\mathcal{I}_1$ in  both sides of \eqref{ineqpot} yields
	$$\mathcal{I}_{2s}\left( \left[ \mathcal{I}_{2s-1}(\omega)\right] ^q\right)\leq C \mathcal{I}_{2s}(\omega)$$Plugging this inequality into \eqref{estimate 23} gives \eqref{est 1 u}. The lower bound in \eqref{est 1 u} is a consequence of \eqref{represent}. The estimate for the gradient \eqref{est grad u} follows from \eqref{to represent}.  Next, we prove the final assertion. Observe $\mathcal{I}_{2s}(\omega)\in L^{q}(\mathbb{R}^{n})$ for $q> n/(n-2s)$. Indeed,  for any $R>0$ so that $\textnormal{supp}(\omega) \subset B_R$ we have
\begin{equation}\label{for the pf}
\int_{B_R}\mathcal{I}_{2s}(\omega)^{q}dx \leq C(n,s,q,R)\int_{B_R}\mathcal{I}_{2s-1}(\omega)^{q}dx < \infty \qquad \text{by Lemma } \ref{regriesz},
\end{equation}	and, on the other hand, for large $R$ and $q> n/(n-2s)$ it follows by Minkowski's inequality
\begin{equation*}
\begin{split}
\int_{\mathbb{R}^{n}\setminus B_R}\mathcal{I}_{2s}(\omega)^{q}dx &\leq C\left( \int_{\text{supp }\omega}\left[\int_{\mathbb{R}^{n}\setminus B_R}\frac{dx}{|x-y|^{q(n-2s)}}\right]^{1/q}d\omega(y)\right) ^q\\ & \leq C\left( \int_{\text{supp }\omega}\left[\int_{\mathbb{R}^{n}\setminus B_R}\frac{dx}{|x|^{q(n-2s)}}\right]^{1/q}d\omega(y)\right) ^q< \infty
\end{split}
\end{equation*}
Consequently by \eqref{est 1 u} $u \in L^{q}(\mathbb{R}^{n})$. Moreover,  by Remark \ref{regI2s-1w} and \eqref{est grad u}, it follows that $\nabla u \in L^{q}(\mathbb{R}^{n})$. Hence $u \in W^{1, q}(\mathbb{R}^{n}).$\end{proof}

The next corollary shows that $u$ from \eqref{represent} is actually a solution of \eqref{problem}.

\begin{corollary}We have the following for the  solution $u$ from Theorem \ref{Theo1}:
\begin{enumerate}
\item[(i)] $u$ vanishes at infinite
$$\lim_{|x|\to \infty}u(x)=0;$$
\item[(ii)] $u$ is positive everywhere.
\end{enumerate}
\end{corollary}
\begin{proof}
Take $R >0$ so that $\textnormal{supp}(\omega) \subset B_R.$ Hence for $|x| > R$ we have by \eqref{est 1 u}
$$u(x) \leq C \mathcal{I}_{2s}(\omega)  = C\int_{B_R}\frac{d \omega(y)}{|x-y|^{n-2s}} \leq C\frac{\omega(B_R)}{\left( |x|-R\right)^{n-2s}}.$$This proves $(i)$. For $(ii)$, observe that for $x \in \mathbb{R}^{n}$ and $r>0$ so that  $\textnormal{supp}(\omega) \subset B_r(x)$ the lower bound in \eqref{est 1 u} implies
$$u(x) \geq c(n, 2s)\int_{B_r(x)} \frac{d\omega(y)}{|x-y|^{n-2s}}\geq c(n, 2s)r^{2s-n}\omega(\textnormal{supp}(\omega))>0.$$
\end{proof}

Observe that from Theorem \ref{Theo1},   $u \in W^{1, q}_{\text{loc}}(\mathbb{R}^{n})$ and thus  $u$ is finite almost everywhere. The following consequence of Proposition \ref{Coro} gives finiteness everywhere for $q > 2s$.\begin{corollary}Suppose that $q> 2s$. Then, $u < \infty$ everywhere in $\mathbb{R}^{n}$. 
\end{corollary}\begin{proof}
	By \eqref{est 1 u} and the Cavalieri's representation of Riesz potentials (see for instance Section 2.2 in \cite{Ponce}), we have for $x \in \mathbb{R}^{n}$
	\begin{equation}\label{finiteness u}
	u(x)\leq C \mathcal{I}_{2s}(\omega)(x) = C \int_0^{\infty}\frac{\omega(B_r(x))}{r^{n-2s+1}}dr= C \int_0^{R(x)}\frac{\omega(B_r(x))}{r^{n-2s+1}}dr,
	\end{equation}where $R(x)>0$ satisfies supp $\omega \subset B(x, R(x))$. We now prove that
	\begin{equation}\label{ineqomegacap}
	\omega(B_r(x))\leq C_2(n,s,q)r^{n-(2s-1) p}.
	\end{equation}For $0 <\alpha < n$, take 		
	$$g=\frac{2^{n-\alpha}}{c(n,\alpha)\omega_n r^\alpha}\mathcal{X}_{B_r(x)}\in L^p_+(\mathbb{R}^n).$$Now, let $z\in B_r(x)$, then
	\begin{equation*}
	\begin{split}
	\mathcal{I}_\alpha g(z)&=c(n,\alpha)\int_{\mathbb{R}^n}\frac{1}{|z-y|^{n-\alpha}}\frac{2^{n-\alpha}}{c(n,\alpha)\omega_n r^\alpha}\mathcal{X}_{B_r(x)}(y)\;dy \\& = \frac{2^{n-\alpha}}{\omega_n r^\alpha}\int_{B_r(x)}\frac{1}{|z-y|^{n-\alpha}}\;dy\\& = 1.
	\end{split}
	\end{equation*}Hence $\mathcal{I}_\alpha g\geq \mathcal{X}_{B_r(x)}$. Therefore,
	\begin{equation*}
	\begin{split}
	\textnormal{cap}_{\alpha,p}(B_r(x))&\leq \|g\|_{L^p(\mathbb{R}^n)}^p\\& = \int_{\mathbb{R}^n}\left( \frac{2^{n-\alpha}}{c(n,\alpha)\omega_n r^\alpha}\right)^p\mathcal{X}_B(y)^p\;dy \\ & = \left( \frac{2^{n-\alpha}}{c(n,\alpha)}\right)^p \omega_n^{1-p}r^{n-\alpha p}.
	\end{split}
	\end{equation*}Thus
	\begin{equation*}
	\textnormal{cap}_{\alpha,p}(B)\leq Cr^{n-\alpha p}, \quad C=\left( \frac{2^{n-\alpha}}{c(n,\alpha)}\right)^p \omega_n^{1-p}.
	\end{equation*}Letting $\alpha = 2s$ and  recalling \eqref{cap1},  we obtain \eqref{ineqomegacap}. Plugging this into \eqref{finiteness u} gives that $u(x)$ is finite (observe that $q > 2s$ implies $p \in (1, 2s/(2s-1)$ and so the last integral in \eqref{finiteness u} is finite).
\end{proof}

The following results show how the regularity of the source is transferred into the regularity of the solution. We start with measure data and then we analyse the regularity of solutions for integrable sources.

To state estimates of the solution in terms of measures, we first recall the definition of the  Marcinkiewicz spaces.
\begin{definition}Let $\Omega \subset \mathbb{R}^{n}$ be a domain and $\mu$ be a positive Borel measure in $\Omega$. For $\kappa > 1$, $\kappa' = \kappa /(\kappa -1)$  we define the Marcinkiewicz space  $M^{\kappa}(\Omega, d \mu)$ of exponent $\kappa$ or weak $L^{\kappa}$-space, as
$$M^{\kappa}(\Omega, d \mu):= \left\lbrace v \in L^{1}_{\text{loc}}(\Omega, d \mu): \|v\|_{M^{\kappa}(\Omega, d \mu)}< \infty \right\rbrace$$where
$$\|v\|_{M^{\kappa}(\Omega, d \mu)}:=\inf\left\lbrace c \in [0, \infty]: \int_E|v|d \mu \leq c \left(\int_E d\mu \right)^{\frac{1}{\kappa'}}, \text{ for all Borel }E \subset \Omega \right\rbrace.$$  \end{definition}The Marcinkiewicz type estimate of the solution $u$ is the following (recall $p^{*}=n/(n-2s+1)$). 
\begin{proposition}There exists $C=C(n, s)>0$ so that
\begin{equation}
\|u\|_{M^{n/(n-2s)}(\mathbb{R}^{n}, d\mu)}+ \|\nabla u\|_{M^{p^{*}}(\mathbb{R}^{n}, d\mu)} \leq C \|\omega\|_{\mathcal{M}^{b}(\mathbb{R}^{n})},
\end{equation}where
$$d\mu(x):= \dfrac{dx}{1+|x|^{n+2s}}$$and $$\|\omega\|_{\mathcal{M}^{b}(\mathbb{R}^{n})}= \omega(\mathbb{R}^{n})$$is a norm in the space of bounded Radon measures $\mathcal{M}^{b}(\mathbb{R}^{n})$ of $\mathbb{R}^{n}$.
\end{proposition}
\begin{proof}We follow closely the proof of \cite[Proposition 2.2]{CV1}. We prove the estimates for $\nabla u$. By similar arguments, the control for $\|u\|_{M^{n/(n-2s)}(\mathbb{R}^{n}, d\mu)}$ is obtained. Observe that, in view of \eqref{est grad u},  it is enough to prove that there is $C>0$ with
\begin{equation}\label{est riesz 2s-1}
 \|\mathcal{I}_{2s-1}(\omega)\|_{M^{n/(n-2s+1)}(\mathbb{R}^{n}, d\mu)} \leq C \|\omega\|_{\mathcal{M}^{b}(\mathbb{R}^{n})}.
\end{equation}For $y \in \mathbb{R}^{n}$ and $\lambda >0$, define
$$A_{\lambda}(y):=\left\lbrace x\in \mathbb{R}^n\setminus \left\lbrace y \right\rbrace: I_{2s-1}(x-y)=\frac{c(n, 2s)}{|x-y|^{n-2s+1}} > \lambda \right\rbrace, \quad m_\lambda(y):=\int_{A_\lambda(y)}\frac{1}{1+|x|^{n+2s}}dx.$$Since
$$A_\lambda \subset B_r(y),$$with $$r=\left[\dfrac{c(n, 2s)}{\lambda}\right]^{1/(n-2s+1)},$$ we have for some $C>0$
$$m_\lambda(y) \leq C \lambda^{-p^{*}}.$$Let now $E$ be a Borel set.  Then
$$\int_E I_{2s-1}(x-y)d\mu(x) \leq \lambda \int_E d \mu(x)+ \int_{A_\lambda(y)}I_{2s-1}(x-y)d\mu(x)$$and
$$\int_{A_\lambda(y)}I_{2s-1}(x-y)d\mu(x) = \lambda m_\lambda(y)+ \int_\lambda^{\infty}m_s(y)ds \leq C \lambda^{1-p^{*}},$$for some $C>0$.  Hence
$$\int_E I_{2s-1}(x-y)d\mu(x) \leq \lambda \int_E d \mu(x)+ C \lambda^{1-p^{*}}.$$Choosing $\lambda = (\int_E d \mu)^{-1/p^{*}}$, we obtain
$$\int_E I_{2s-1}(x-y)d\mu(x) \leq C\left(\int_E d\mu \right)^{\frac{p^{*}-1}{p^{*}}},$$for a universal constant $C>0$ and all $y \in \mathbb{R}^{n}.$ As a result
$$\int_E \mathcal{I}_{2s-1}(\omega)(x)d\mu(x) \leq C \|\omega \|_{\mathcal{M}^{b}(\mathbb{R}^{n})}\left( \int_E d\mu\right)^{\frac{p^{*}-1}{p^{*}}}.$$This ends the proof of the proposition. 
\end{proof}


In the next proposition, we provide a complete scheme of Lebesgue, Sobolev and H\"{o}lder regularity of $u$ in terms of the regularity of the source.

\begin{proposition}
Let  $\omega= f dx$, where $f\in L^{m}_{+}(\mathbb{R}^{n})$ vanishes outside a compact set. Then for the  solution $u$ from Theorem \ref{Theo1} we get:
\begin{itemize}
\item[(i)] if $m=1$, then there is $C=C(s, n)>0$ so that  for any $\lambda> 0$
$$|\left\lbrace x \in \mathbb{R}^{n}: u(x) > \lambda\right\rbrace| \leq C\left(\dfrac{\|f\|_{L^{1}(\mathbb{R}^{n})}}{\lambda}\right)^{m*}, \qquad m^{*}= 1-\frac{2s}{n};$$
\item[(ii)] if $m >1 $ and $2sm < n$, there is a constant $C=C(r, s,n, m)>0$ so that
$$\|u\|_{L^{r}(\mathbb{R}^{n})} \leq C\|f\|_{L^{m}(\mathbb{R}^{n})}, \quad \text{ for all  }r \in \left( \frac{n}{n-2s}, m^{*}\right]$$where $$m^{*}= \frac{nm}{n-2sm};$$\item[(iii)] for $m >1 $ and $(2s-1)m < n$, for some $C=C(r, s,n, m)>0$ we have
$$\|\nabla u\|_{L^{r}(\mathbb{R}^{n})} \leq C\|f\|_{L^{m}(\mathbb{R}^{n})}, \quad \text{ for all  }r \in \left( \frac{n}{n-2s+1}, m^{*}\right]$$where
$$m^{*}= \frac{nm}{n-(2s-1)m};$$\item[(iv)] if $m=n/(2s)$, then $$u \in L^{r}(\mathbb{R}^{n})\quad \text{ for all  }r \in \left( \frac{n}{n-2s}, \infty\right);$$

 
 \item[(v)]in the case $m=n/(2s-1)$,  $$u \in W^{1, r}(\mathbb{R}^{n})\quad \text{ for all  }r \in \left( \frac{n}{n-2s+1}, \infty\right)$$

\item[(vi)] if $m > n/(2s-1)$, then $$\nabla u \in \mathcal{C}^{0, \gamma}(\mathbb{R}^{n}), \quad\gamma = 2s-1 - \frac{n}{m}.$$  
\end{itemize}
\end{proposition}\begin{proof}
The estimates $(i)-(iii)$ are direct consequences of Proposition \ref{Coro} and well-known $L^{p}$ embeddings of the Riesz potential (see for instance \cite[Chapter V]{St} and \cite[Chapter 3]{AH}). Also, $(iv)$ and $(v)$ follow from $(ii)$ and $(iii)$, respectively,  together with the assumption that $f$ has compact support.  Finally, for $(vi)$ we first observe that  from \eqref{represent}
$$\nabla u = \mathcal{I}_{2s-1}(|\nabla u|^{q}+ f) \quad a.e. \,\,\mathbb{R}^{n}.$$Moreover, $\nabla u \in L^{qm}(\mathbb{R}^{n})$ since by $(v)$, $u \in W^{1, r}(\mathbb{R}^{n})$ for all $r > p^{*}$ and $qm > p^{*}$.  Therefore, appealing to  \cite[Theorem 2.2, Sec. 4.2]{Mi}, there is a constant $M>0$ so that
$$|\nabla u(x)-\nabla u(y)| \leq M|x-y|^{2s-1 - n/m}\||\nabla u|^{q}+ f\|_{L^{m}(\mathbb{R}^{n})}\,\,\,\, a.e. \text{ in }\mathbb{R}^{n}.$$
\end{proof}

By \cite[Theorem 3.2.1]{AH}, we also may obtain the following exponential summability for $u$ and $\nabla u$ which account, for instance, to local integrability: 
\begin{itemize}
\item if $m=n/2s$ and supp $f \subset B_R$, there is a constant $A=A(n, m)$ such that
$$\int_{B_R}\exp\left(A_0 u^{m'}\right) dx \leq A R^{n}, \quad A_0=\frac{n}{ c(n, 2s)^{m'}\omega_{n-1}};$$
\item  in the case $m=n/(2s-1)$ and supp $f \subset B_R$, we have for the same $A$ as before that
$$\int_{B_R}\exp\left(A_0 |\nabla u|^{m'}\right) dx \leq A R^{n}, \quad A_0=\frac{n}{ c(n, 2s-1)^{m'}\omega_{n-1}}.$$
\end{itemize}

\begin{remark}The interested reader may compare  the above regularity results to the related findings for bounded domains and no first-order terms presented in \cite[Theorem 15-16]{LP} and \cite[Lemma 2.15]{AyP}.
\end{remark}

\section{Proof of Theorem \ref{Theo1}}\label{proof}

The structure of the proof  consists of the following steps
\begin{itemize}
\item[(I)] the starting point will be to consider $u_0=\mathcal{I}_{2s}(\omega)$ and prove that
$$(-\Delta)^{s}u_0=\omega$$in the sense of Definition \ref{Defsolomega}. This will be done in Lemma \ref{regI2sw} and Proposition \ref{SolRiesz}. 
\item[(II)] The next step is to consider first-order terms. Indeed, we show in  Lemmas \ref{ineqgrad}-\ref{regriesz}, and Propositions \ref{propp}-\ref{solu0} that $v= \mathcal{I}_{2s}(|\nabla u_0|^{q})$ solves
$$(-\Delta)^{s}v=|\nabla u_0|^{q}.$$
\item[(III)]  The final step, developed at the end of the section, is to define by recursion the sequence
$$u_{k+1}= \mathcal{I}_{2s}(\omega)+ \mathcal{I}_{2s}(|\nabla u_k|^{q}), \quad k\geq 0,$$and prove that $u_k$ converges in the right topology to a solution of  \eqref{eq}. 
\end{itemize}

\subsection{Step (I)}

\begin{lemma}\label{regI2sw}
	
	Let $\omega$ be a nonnegative Radon measure with compact support in $\mathbb{R}^n$. Then, for $n>2s$, $\mathcal{I}_{2s}(\omega)\in L_s(\mathbb{R}^n).$ 
	
\end{lemma}

\begin{proof}
	
	First note that $\mathcal{I}_{2s}(\omega)\in L^1_{\textnormal{loc}}(\mathbb{R}^n)$. Let $R>0$ such that $\textnormal{supp}(\omega)\subset B_R$. For simplicity take $R=1$. Then, 
	
	\begin{equation*}
	\begin{split}
	\int_{\mathbb{R}^n}\frac{|\mathcal{I}_{2s}(\omega)(x)|}{1+|x|^{n+2s}}\; dx  & = \int_{B_1}\left( \int_{\mathbb{R}^n}\frac{c(n,2s)}{(1+|x|^{n+2s})|x-y|^{n-2s}}\;dx\right)d\omega(y). 
	\end{split}
	\end{equation*}Now we split the last integral as follows (we omit the constant for simplicity)
	
	\begin{equation*}
	\begin{split}
	\int_{B_1}\left( \int_{\mathbb{R}^n}\frac{1}{(1+|x|^{n+2s})|x-y|^{n-2s}}\;dx\right)d\omega(y) &=\int_{B_1}\left( \int_{B_2}\frac{1}{(1+|x|^{n+2s})|x-y|^{n-2s}}\;dx\right)d\omega(y) \\&+ \int_{B_1}\left( \int_{\mathbb{R}^n\setminus B_2}\frac{1}{(1+|x|^{n+2s})|x-y|^{n-2s}}\;dx\right)d\omega(y).
	\end{split}
	\end{equation*}Then, for the first integral we have
	
	\begin{equation*}
	\begin{split}
	\int_{B_1}\left(\int_{B_2}\frac{1}{(1+|x|^{n+2s})|x-y|^{n-2s}}\;dx\right)d\omega(y)&\leq \int_{B_1}\left( \int_{B_2}\frac{1}{|x-y|^{n-2s}}\;dx\right)d\omega(y)\\& \leq \int_{B_1}\left( \int_{B_{2+|y|}}|z|^{2s-n}\;dz\right)d\omega(y)\\ & = \omega_n \int_{B_1}\frac{(2+|y|)^{2s}}{2s}\;d\omega(y) \\ & <\infty. 
	\end{split}
	\end{equation*}For the second integral we have
	
	
	\begin{equation*}
	\begin{split}
	\int_{B_1}\left( \int_{\mathbb{R}^n\setminus B_2}\frac{1}{(1+|x|^{n+2s})|x-y|^{n-2s}}\,dx\right)d\omega(y)& \leq \int_{B_1}\left(\int_{\mathbb{R}^n\setminus B_2}(|x|-|y|)^{2s-n}|x|^{-n-2s}\,dx \right)d\omega(y)\\& <\infty
	\end{split}	
	\end{equation*}
	
\end{proof}

\begin{remark}\label{regI2s-1w}
	\rm{Reproducing the proof of Lemma \ref{regI2sw} with $2s-1$ instead of $2s$, one obtains $\mathcal{I}_{2s-1}(\omega) \in L_s(\mathbb{R}^{n})$ as well.}
\end{remark}



%
%
%
%
%
%

For the proof of Proposition \ref{SolRiesz} is necessary the following lemma (Proposition 2.4 of \cite{Buc}).
	
\begin{lemma}\label{LemmaBucur}
	
Let $n>2s$ and $f\in L^1(\mathbb{R}^n)\cap \mathcal{C}(\mathbb{R}^n)$ with $\mathcal{F}^{-1}(f)\in S_s(\mathbb{R}^n)$. Then
		\begin{equation}\label{eqprop2.4}
		\int_{\mathbb{R}^n}I_{2s}(x)\mathcal{F}^{-1}(f)(x)\;dx=\int_{\mathbb{R}^n}|x|^{-2s}f(x)\;dx.
		\end{equation}
		
\end{lemma}

\begin{remark}\label{rm Buc} A closer look at the proof of \cite[Proposition 2.4]{Buc} reveals that the assumption $\mathcal{F}^{-1}(f)\in S_s(\mathbb{R}^n)$ may be replaced by \begin{enumerate}[i.]
		\item $\int_{\mathbb{R}^n}I_{2s}(x)\mathcal{F}^{-1}(f)(x)\,dx<\infty$ \rm{and}
		\item $\mathcal{F}^{-1}(f)\in L^2(\mathbb{R}^n)$.
	\end{enumerate}We will use this observation in Proposition \ref{solu0}.
\end{remark}

%



\bigskip 

\begin{proposition}\label{SolRiesz}
	
	For $n>2s$, $\mathcal{I}_{2s}(\omega)$ is a weak solution of
		
		\begin{equation}\label{eq2}
		(-\Delta)^s u = \omega 	\quad \text{in }\mathbb{R}^n.
		\end{equation}

\end{proposition}

\begin{proof}

We will prove that Definition \ref{Defsolomega} is satisfied. We already know that $\mathcal{I}_{2s}(\omega)\in L_s(\mathbb{R}^n)$ by Lemma \ref{regI2sw}. Let $\varphi_0\in\mathcal{S}(\mathbb{R}^n)$. Put $$\varphi(x)=|x|^{2s}\mathcal{F}(\varphi_0)(x)$$ and take $\psi$ such that $$\varphi(x)=\mathcal{F}(\psi)(-x).$$Then, $\psi(x)=\mathcal{F}^{-1}(\varphi)(-x)$. Thus, since $\varphi_0\in \mathcal{S}(\mathbb{R}^n)$, $\mathcal{F}^{-1}(\varphi)\in \mathcal{S}_s(\mathbb{R}^n)$ and therefore $\psi\in \mathcal{S}_s(\mathbb{R}^n)$.

In what follows, we shall employ Lemma \ref{LemmaBucur}.  We first prove that $\psi*\omega\in\mathcal{S}_s(\mathbb{R}^n)$. Since $D^\alpha \psi\in\mathcal{S}_s(\mathbb{R}^n)$ and $D^\alpha(\psi*\omega)=(D^\alpha \psi)*\omega$ we just need to verify that 

\begin{equation}\label{cotapsiconvomega}
|\psi*\omega(x)|\leq C(1+|x|^{n+2s})^{-1}\qquad\qquad x\in\mathbb{R}^n.
\end{equation}

Take $R>0$ such that $\textnormal{supp}(\omega)\subset B_R$. 

\begin{equation*}
\begin{split}
|\psi*\omega(x)|& \leq \int_{B_R}|\psi(x-y)|\,d\omega(y)\leq C\int_{B_R}(1+|x-y|^{n+2s})^{-1}\,d\omega(y) \\& = C(1+|x|^{n+2s})^{-1}\int_{B_R}\frac{1+|x|^{n+2s}}{1+|x-y|^{n+2s}}\,d\omega(y).
\end{split}
\end{equation*}Now, observe that $$\frac{|x|^{n+2s}}{1+(|x|-R)^{n+2s}}\to 1\qquad\qquad \textnormal{for}\;|x|\to\infty.$$Then, there is $r>0$ such that 
\begin{equation}\label{lim int}
\left|\frac{|x|^{n+2s}}{1+(|x|-R)^{n+2s}}-1 \right|<1
\end{equation}for $|x|>r$. Let $r_0>\max\{2R,r\}$. Then, for $x\in B_{r_0}$

$$\int_{B_R}\frac{1+|x|^{n+2s}}{1+|x-y|^{n+2s}}\,d\omega(y)\leq (1+r_0^{n+2s})\omega(B_R)<\infty$$and for $x\in\mathbb{R}^n\setminus B_{r_0}$,

\begin{equation*}
\begin{split}
\int_{B_R}\frac{1+|x|^{n+2s}}{1+|x-y|^{n+2s}}\,d\omega(y)& = \int_{B_R}\frac{1}{1+|x-y|^{n+2s}}\,d\omega(y) + \int_{B_R}\frac{|x|^{n+2s}}{1+|x-y|^{n+2s}}\,d\omega(y) \\& \leq \omega(B_R) + \int_{B_R}\frac{|x|^{n+2s}}{1+|x-y|^{n+2s}}\,d\omega(y)\\&<\infty	
\end{split}
\end{equation*}by \eqref{lim int} and the choice of $r_0$. Thus, we have proved \eqref{cotapsiconvomega} and we  conclude that $\psi*\omega\in \mathcal{S}_s(\mathbb{R}^n)$. 

Moreover, observe that the integrability of $\psi*\omega$ implies $\mathcal{F}(\psi*\omega)\in\mathcal{C}^{\infty}(\mathbb{R}^n).$

%

To verify the integrability of $\mathcal{F}(\psi*\omega)$ note that it easily follows
$$\mathcal{F}(\psi*\omega)(x) = \mathcal{F}(\psi)(x)\mathcal{F}(\omega)(x) \qquad  \text{ for all }x.$$

%
Recalling $\mathcal{F}(\psi)(x)=\varphi(-x)=|x|^{2s}\mathcal{F}(\varphi_0)(-x)$, we have

\begin{equation*}
\begin{split}
\int_{\mathbb{R}^n}|\mathcal{F}(\psi*\omega)(x)|\,dx & = \int_{\mathbb{R}^n}|\mathcal{F}(\psi)(x)||\mathcal{F}(\omega)(x)|\,dx \\& \leq\omega(B_R)\int_{\mathbb{R}^n}|\mathcal{F}(\psi)(x)|\,dx \\& = \omega(B_R)\int_{\mathbb{R}^n}|x|^{2s}|\mathcal{F}(\varphi_0)(-x)|\,dx \\& = \omega(B_R)\left[\int_{B_1}|x|^{2s}|\mathcal{F}(\varphi_0)(-x)|\,dx + \int_{\mathbb{R}^n\setminus B_1}|x|^{2s}|\mathcal{F}(\varphi_0)(-x)|\,dx\right] \\ & \leq C+  \int_{\mathbb{R}^n\setminus B_1}|x|^{2s}(1+|x|)^{-n-2}\,dx < \infty.
\end{split}
\end{equation*}


%

We can apply now Lemma \ref{LemmaBucur} to $f=\mathcal{F}(\psi*\omega)$ to obtain

\begin{equation}\label{1 eqq}
\begin{split}
 \int_{\mathbb{R}^n}I_{2s}(x)(\psi*\omega)(x)\;dx  &=\int_{\mathbb{R}^n}I_{2s}(x)\mathcal{F}^{-1}(\mathcal{F}(\psi*\omega))(x)\;dx\\& = \int_{\mathbb{R}^n}|x|^{-2s}\mathcal{F}(\omega)(x)\mathcal{F}(\psi)(x)\;dx.
\end{split} 
\end{equation}Now,

\begin{equation}\label{eq 1q}
\begin{split}
\int_{\mathbb{R}^n}I_{2s}(x)(\psi*\omega)(x)\;dx &= \int_{\mathbb{R}^n}\int_{\mathbb{R}^n}I_{2s}(x)\psi(x-y)\,d\omega(y)\,dx \\& = \int_{\mathbb{R}^n}\left[\int_{\mathbb{R}^n} I_{2s}(x)\mathcal{F}^{-1}(\varphi)(y-x)\,dx\right]\,d\omega(y). 
\end{split}
\end{equation}We make the change of variable $z=x-y$ in the last integral and obtain

\begin{equation}\label{eq 1}
\begin{split}
\int_{\mathbb{R}^n}\left[ \int_{\mathbb{R}^n}I_{2s}(y+z)\mathcal{F}^{-1}(\varphi)(-z)\,dz\right] \,d\omega(y) & = \int_{\mathbb{R}^n}\mathcal{F}^{-1}(\varphi)(-z)\left[ \int_{\mathbb{R}^n}I_{2s}(y+z)\,d\omega(y)\right]\,dz \\& = \int_{\mathbb{R}^n}\mathcal{I}_{2s}(\omega)(-z)\mathcal{F}^{-1}(\varphi)(-z)\,dz \\ & = -\int_{\mathbb{R}^n}\mathcal{I}_{2s}(\omega)(x)\mathcal{F}^{-1}(\varphi)(x)\,dx. 
\end{split}
\end{equation}Thus from  \eqref{1 eqq}, \eqref{eq 1q} and \eqref{eq 1} we get
\begin{equation}\label{eqq12}
\int_{\mathbb{R}^n}\mathcal{I}_{2s}(\omega)(x)\mathcal{F}^{-1}(\varphi)(x)\,dx =- \int_{\mathbb{R}^n}|x|^{-2s}\mathcal{F}(\omega)(x)\mathcal{F}(\psi)(x)\;dx =- \int_{\mathbb{R}^n}|x|^{-2s}\mathcal{F}(\omega)(x)\varphi(-x)\;dx.
\end{equation}

Then

\begin{equation*}
\begin{split}
\int_{\mathbb{R}^n}\mathcal{I}_{2s}(\omega)(x)(-\Delta)^s(\varphi_0)(x)\,dx & = -\int_{\mathbb{R}^n}|x|^{-2s}\mathcal{F}(\omega)(x)\varphi(-x)\,dx  \qquad \text{by }\eqref{eqq12}\\& = \int_{\mathbb{R}^n}|x|^{-2s}\mathcal{F}(\omega)(-x)\varphi(x)\,dx \\& = \int_{\mathbb{R}^n}|x|^{-2s}\mathcal{F}^{-1}(\omega)(x)|x|^{2s}\mathcal{F}(\varphi_0)(x)\,dx \\& = \int_{\mathbb{R}^n}\left[ \int_{\mathbb{R}^n}e^{2\pi ix\cdot y}\,d\omega(y)\right] \mathcal{F}(\varphi_0)(x)\,dx \\& = \int_{\mathbb{R}^n}\left[ \int_{\mathbb{R}^n}e^{2\pi ix\cdot y}\mathcal{F}(\varphi_0)(x)\,dx\right] d\omega(y)\\& = \int_{\mathbb{R}^n}\varphi_0(y)\,d\omega(y).
\end{split}
\end{equation*}

\end{proof}

\subsection{Step (II)}

\begin{lemma}\label{ineqgrad}
	
	Let $u_0=\mathcal{I}_{2s}(\omega)$. There exists $C_0=C_0(n, s)$  such that
	
	\begin{equation}\label{eqgrad}
	|\nabla u_0|\leq C_0\mathcal{I}_{2s-1}(\omega).
	\end{equation}
	
\end{lemma}

\begin{proof}
	
We will show that 
\begin{equation}\label{deriv}
\frac{\partial u_0}{\partial x_i}(x)=c(n,2s)\int_{\mathbb{R}^n}\frac{(x_i-y_i)}{|x-y|^{n-2s+2}}\;d\omega(y)\qquad i=1,...,n.
\end{equation}in the weak sense. Let $\varphi\in \mathcal{C}_0^\infty(\mathbb{R}^n)$ and let $R>0$ such that $\textnormal{supp}(\varphi)\subset B_R$. We want to show

\begin{equation}\label{weakderiv}
\int_{\mathbb{R}^n}u_0(x)\frac{\partial\varphi}{\partial x_i}(x)\;dx = -\int_{\mathbb{R}^n}\left[c(n,2s)\int_{\mathbb{R}^n}\frac{(x_i-y_i)}{|x-y|^{n-2s+2}}\;d\omega(y) \right]\varphi(x)\;dx. 
\end{equation}Now, by Fubini's Theorem  we have that

\begin{equation*}
\begin{split}
& \int_{\mathbb{R}^n}\left[ \int_{\textnormal{supp}(\omega)}\frac{c(n,2s)}{|x-y|^{n-2s}}\frac{\partial\varphi}{\partial x_i}(x)\;d\omega(y)\right]dx  \\ & \qquad \qquad =  \int_{\textnormal{supp}(\omega)}\left[ \int_{B_R}\frac{c(n,2s)}{|x-y|^{n-2s}}\frac{\partial\varphi}{\partial x_i}(x)\;dx\right]d\omega(y) 
\end{split}
\end{equation*}and

\begin{equation*}
\begin{split}
&-\int_{\mathbb{R}^n}\left[ c(n,2s)\int_{\textnormal{supp}(\omega)}\frac{(x_i-y_i)}{|x-y|^{n-2s+2}}\varphi(x)d\omega(y)\right]dx \\ & \qquad \qquad  = -c(n,2s)\int_{\textnormal{supp}(\omega)}\left[ \int_{B_R}\frac{(x_i-y_i)}{|x-y|^{n-2s+2}}\varphi(x)dx\right]d\omega(y).  
\end{split}
\end{equation*}For $r>0$ small and $y\in\textnormal{supp}(\omega)$ fixed, integration by parts gives 

\begin{equation}\label{int by parts}
\int_{B_R\setminus B_r(y)}\frac{1}{|x-y|^{n-2s}}\frac{\partial \varphi}{\partial x_i}(x)\;dx  = \int_{\partial B_r(y)}\frac{\varphi(x)\eta_i}{|x-y|^{n-2s}}\;dS_x - \int_{B_R\setminus B_r(y)}\frac{(x_i-y_i)}{|x-y|^{n-2s+2}}\varphi(x)\;dx. 
\end{equation}Here we have used the fact that $\varphi=0$ in $\partial B_R$. The vector $\eta=(\eta_1,...,\eta_n)$ is the exterior normal unit vector to $\partial B_r(y)$, so $$\eta_i=\frac{(x_i-y_i)}{|x-y|}.$$

Now, for all $y\in\textnormal{supp}(\omega)$

\begin{equation*}
\begin{split}
\left|\int_{\partial B_r(y)}\frac{1}{|x-y|^{n-2s}}\varphi(x)\eta_i\;dS_x \right|=&\left| \int_{\partial B_r(y)}\frac{(x_i-y_i)}{|x-y|^{n-2s+1}}\varphi(x)\;dS_x\right| \\&\leq  \int_{\partial B_r(y)}\frac{1}{|x-y|^{n-2s}}|\varphi(x)|\;dS_x \\ & = r^{2s-n}\int_{\partial B_r(y)}|\varphi(x)|\;dS_x \\& = r^{2s-1}r^{1-n}\int_{\partial B_r(y)}|\varphi(x)|\;dS_x.
\end{split}
\end{equation*}Observe

$$\lim_{r\to 0}r^{1-n}\int_{\partial B_r(y)}|\varphi(x)|\;dS_x = \omega_n|\varphi(y)|\leq C$$for all $y\in\textnormal{supp}(\omega).$ So

$$\int_{\partial B_r(y)}\frac{1}{|x-y|^{n-2s}}\varphi(x)\eta_i \;dS_x= o(1)$$for $r\to 0$ uniformly on $y\in\textnormal{supp}(\omega).$ Regarding the last term in \eqref{int by parts} we have

$$\int_{B_R\setminus B_r(y)}\frac{(x_i-y_i)}{|x-y|^{n-2s+2}}\varphi(x)\;dx \leq  C \int_{B_R}\frac{1}{|x-y|^{n-2s+1}}dx$$and taking $R_0>0$ big enough we obtain that

\begin{equation*}
\begin{split}
\int_{\textnormal{supp}(\omega)}\int_{B_{R_0}}\frac{1}{|x-y|^{n-2s+1}}\;dx\;d\omega(y)& = \int_{\textnormal{supp}(\omega)}\int_{0}^{R_0}\omega_n r^{2s-n-1}r^{n-1}\;dr\;d\omega(y)\\& \leq C R_0^{2s-1}<\infty.
\end{split}
\end{equation*}
%

\noindent Hence, by Lebesgue dominated convergence theorem

$$\lim_{r\to 0}\int_{\textnormal{supp}(\omega)}\int_{B_R}\frac{(x_i-y_i)}{|x-y|^{n-2s+2}}\varphi(x)\mathcal{X}_{B_R\setminus B_r(y)}\;dx\;d\omega(y)= \int_{\textnormal{supp}(\omega)}\int_{B_R}\frac{(x_i-y_i)}{|x-y|^{n-2s+2}}\varphi(x)\;dx\;d\omega(y).$$Therefore

$$\lim_{r\to 0}\left[ \int_{\textnormal{supp}(\omega)}\int_{B_R\setminus B_r(y)}\frac{1}{|x-y|^{n-2s}}\frac{\partial\varphi}{\partial x_i}(x)\;dx\;d\omega(y)\right]= -\int_{\textnormal{supp}(\omega)}\int_{B_R}\frac{(x_i-y_i)}{|x-y|^{n-2s+2}}\varphi(x)\;dx\;d\omega(y).$$

Similarly

%

$$\lim_{r\to 0}\int_{\textnormal{supp}(\omega)}\int_{B_R\setminus B_r(y)}\frac{1}{|x-y|^{n-2s}}\frac{\partial \varphi}{\partial x_i}\;dx\;d\omega(y) = \int_{\textnormal{supp}(\omega)}\int_{B_R}\frac{1}{|x-y|^{n-2s}}\frac{\partial\varphi}{\partial x_i}\;dx\;d\omega(y).$$Then, we have proved \eqref{weakderiv} and therefore we have $$\nabla u_0(x) = c(n,2s)\int_{\mathbb{R}^n}\frac{(x-y)}{|x-y|^{n-2s+2}}\;d\omega(y)$$in the weak sense.    	
	
\end{proof}

\begin{lemma}\label{regriesz}
For $n>2s$ and $q>p^*$, if hypothesis \eqref{ineqpot} is satisfied with a constant $C_1=C_1(n, q, s)$, then $\mathcal{I}_{2s-1}(\omega)\in L^q(\mathbb{R}^n)$.
\end{lemma}

\begin{proof}
	
First of all, by Theorem 2.1 from \cite{MV} and the assumption \eqref{ineqpot} we have that

\begin{equation}\label{rieszLqloc}
\mathcal{I}_{2s-1}(\omega)\in L^q_{\textnormal{loc}}(\mathbb{R}^n).
\end{equation}

Let $R>0$ such that $\textnormal{supp}(\omega)\subset B_R$. We write

\begin{equation*}
\int_{\mathbb{R}^n}|\mathcal{I}_{2s-1}(\omega)(x)|^q\,dx = \int_{B_{2R}}|\mathcal{I}_{2s-1}(\omega)(x)|^q\,dx + \int_{\mathbb{R}^n\setminus B_{2R}}|\mathcal{I}_{2s-1}(\omega)(x)|^q\,dx.
\end{equation*}Then, by \eqref{rieszLqloc},

\begin{equation*}
\int_{B_{2R}}|\mathcal{I}_{2s-1}(\omega)(x)|^q\,dx  <\infty.
\end{equation*}On the other hand, using Minkowsky's inequality

\begin{equation*}
\int_{\mathbb{R}^n\setminus B_{2R}}|\mathcal{I}_{2s-1}(\omega)(x)|^q\,dx \leq C(n,2s-1)^q\left[\int_{B_R}\left(\int_{\mathbb{R}^n\setminus B_{2R}}\frac{1}{|x-y|^{q(n-2s+1)}}\,dx \right) ^{\frac{1}{q}}\,d\omega(y) \right] ^q
\end{equation*} 
Now, since $q>\frac{n}{n-2s+1}$
\begin{equation*}
\int_{\mathbb{R}^n\setminus B_{2R}}\frac{1}{|x-y|^{q(n-2s+1)}} \,dx  \leq \frac{R^{n+(2s-n-1)q}}{(n-2s+1)q -n}
\end{equation*}and so the integral 

\begin{equation*}
\left[\int_{B_R}\left( \int_{\mathbb{R}^n\setminus B_{2R}} \frac{1}{|x-y|^{q(n-2s+1)}} \,dx \right)^{1/q}d\omega(y)  \right]^q
\end{equation*}is finite.

\end{proof}

\begin{remark}
\rm{ Observe that by Lemma \ref{ineqgrad}} and Lemma \ref{regriesz}, $|\nabla u_0|\in L^q(\mathbb{R}^n)$. 	
\end{remark}

\begin{proposition}\label{propp}
	Under the same conditions of Lemma \ref{regriesz}, $\mathcal{I}_{2s}(|\nabla u_0|^q)\in L_s(\mathbb{R}^n)$.
\end{proposition}

\begin{proof}
	Using the inequality \eqref{eqgrad}, we just need to verify that $\mathcal{I}_{2s}(\left[ \mathcal{I}_{2s-1}(\omega)\right] ^q)\in L_s(\mathbb{R}^n)$. Hence
	
	\begin{equation*}
	\begin{split}
	\mathcal{I}_{2s}(\left[ \mathcal{I}_{2s-1}(\omega)\right] ^q)(x) & = \int_{\mathbb{R}^n}\frac{c(n,2s)}{|x-y|^{n-2s}}\left[ \mathcal{I}_{2s-1}(\omega)\right]^q(y)\,dy \\& = c(n,2s) \left(\int_{B_1(x)}\frac{\left[ \mathcal{I}_{2s-1}(\omega)\right]^q(y)}{|x-y|^{n-2s}}\,dy + \int_{\mathbb{R}^n\setminus B_1(x)}\frac{\left[ \mathcal{I}_{2s-1}(\omega)\right]^q(y)}{|x-y|^{n-2s}}\,dy\right) \\& \leq c(n,2s) \left(\int_{B_1(x)}\frac{\left[ \mathcal{I}_{2s-1}(\omega)\right]^q(y)}{|x-y|^{n-2s+1}}\,dy + \int_{\mathbb{R}^n\setminus B_1(x)}\left[ \mathcal{I}_{2s-1}(\omega)\right]^q(y)\,dy\right)  \\& = C(n,s)\left[ \mathcal{I}_{2s-1}(\left[ \mathcal{I}_{2s-1}(\omega)\right] ^q)(x) + \|\mathcal{I}_{2s-1}(\omega)\|^q_{L^q(\mathbb{R}^n)}\right] \\& \leq C(q,n,s)\left[ \mathcal{I}_{2s-1}(\omega)(x) + \|\mathcal{I}_{2s-1}(\omega)\|^q_{L^q(\mathbb{R}^n)}\right] .
	\end{split}
	\end{equation*}Thus, using Remark \ref{regI2s-1w} we conclude the proof. 
\end{proof}

\begin{proposition}\label{solu0}
	For $n>2s$ and $q>\frac{n}{n-2s+1}$, if there exists  $C_1=C_1(n, q, s)>0$  such that \eqref{ineqpot} is satisfied, then $\mathcal{I}_{2s}(|\nabla u_0|^q)$ is a weak solution of
	\begin{equation}\label{eq3}
	(-\Delta)^sv=|\nabla u_0|^q \quad \text{in }\mathbb{R}^{n}.
	\end{equation}
\end{proposition}

\begin{proof}
Let $\varphi_0\in \mathcal{S}(\mathbb{R}^n)$. We proceed as in Proposition \ref{SolRiesz}, naming $\varphi(x)=|x|^{2s}\mathcal{F}(\varphi_0)(x)$ and $\psi(x)=\mathcal{F}(\varphi)(-x)$. Recall that $\psi\in\mathcal{S}_s(\mathbb{R}^n)$.

We will apply Lemma \ref{LemmaBucur} as in the proof of Lemma \ref{SolRiesz} to obtain the desired result. According to Remark \ref{rm Buc}, we just need to  verify the following statements

\

\begin{enumerate}[i.]
\item $\mathcal{F}(\psi*|\nabla u_0|^q)\in L^1(\mathbb{R}^n)\cap\mathcal{C}(\mathbb{R}^n)$,

\

	\item $\int_{\mathbb{R}^n}I_{2s}(x)(\psi*|\nabla u_0|^q)(x)\,dx<\infty$, and
	
	\
	
	\item $\psi*|\nabla u_0|^q\in L^2(\mathbb{R}^n)$.
\end{enumerate}

\

First of all, $\mathcal{F}(\psi*|\nabla u_0|^q)\in L^1(\mathbb{R}^n)\cap\mathcal{C}(\mathbb{R}^n)$. The continuity follows from the fact that $\psi*|\nabla u_0|^q\in L^1(\mathbb{R}^n)$ since $\psi, |\nabla u_0|^q \in L^1(\mathbb{R}^n)$. The integrability can be checked easily using that $|\nabla u_0|^q\in L^1(\mathbb{R}^n)$. Thus, we will concentrate on  ii. and iii.  

ii. 

\begin{equation*}
\begin{split}
\int_{\mathbb{R}^n}|x|^{2s-n}(\psi*|\nabla u_0|^q)(x)\,dx &= \int_{\mathbb{R}^n}\frac{1}{|x|^{n-2s}}\left(\int_{\mathbb{R}^n}\psi(x-y)|\nabla u_0(y)|^q\,dy \right)\,dx \\& = \int_{\mathbb{R}^n}|\nabla u_0(y)|^q\left( \int_{\mathbb{R}^n}\frac{\psi(x-y)}{|x|^{n-2s}}\,dx\right)\,dy \\& = \int_{\mathbb{R}^n}|\nabla u_0(y)|^q\left(\int_{B_1}\frac{\psi(x-y)}{|x|^{n-2s}}\,dx + \int_{\mathbb{R}^n\setminus B_1}\frac{\psi(x-y)}{|x|^{n-2s}}\,dx \right)\,dy.   
\end{split}
\end{equation*}Now, since $\psi\in\mathcal{S}_s(\mathbb{R}^n)$,

\begin{equation*}
\int_{B_1}\frac{\psi(x-y)}{|x|^{n-2s}}\,dx\leq C \int_{B_1}\frac{1}{|x|^{n-2s}}\,dx = \frac{C\omega_n}{2s}.
\end{equation*}On the other hand,

\begin{equation*}
\int_{\mathbb{R}^n\setminus B_1}\frac{\psi(x-y)}{|x|^{n-2s}}\,dx\leq \|\psi\|_{L^1(\mathbb{R}^n)}.
\end{equation*}Then, 

\begin{equation*}
\int_{\mathbb{R}^n}|x|^{2s-n}(\psi*|\nabla u_0|^q)(x)\,dx\leq C(n,s,\psi)\|\nabla u_0\|_{L^q(\mathbb{R}^n)}<\infty.
\end{equation*}

iii. Follows by Young's inequality.

\end{proof}

\subsection{Step (III)}

\begin{proof}[Proof of Theorem \ref{Theo1}]
	
	We begin constructing a sequence of functions $u_{k+1}$ as follows.
	Let $u_0$ be as in Lemma \ref{ineqgrad} and define by recursion
	\begin{equation}\label{seq}
	u_{k+1}=\mathcal{I}_{2s}(|\nabla u_k|^q)+\mathcal{I}_{2s}(\omega).
	\end{equation}Then by propositions \ref{SolRiesz} and  \ref{solu0},
	
	\begin{equation}
	(-\Delta)^su_{k+1}=|\nabla u_k|^q+\omega \quad \text{in }\mathbb{R}^{n}. 
	\end{equation}We claim that
	
	\begin{equation}\label{ineq1}
	|\nabla u_k|\leq C_2\mathcal{I}_{2s-1}(\omega),	
	\end{equation}
	
	\begin{equation}\label{ineq2}
	|\nabla u_{k+1}-\nabla u_k|\leq C_3\delta^k\mathcal{I}_{2s-1}(\omega)	
	\end{equation}for some $0<\delta<1$. Let us begin with \eqref{ineq1} proceding by induction. Beacuse of Lemma \ref{ineqgrad}, \eqref{ineq1} holds for $k=0$. Suppose that
\begin{equation}\label{induc}
|\nabla u_k|\leq a_k\mathcal{I}_{2s-1}(\omega).
\end{equation}Then, we see that
	
	\begin{equation*}
	\begin{split}
	|\nabla u_{k+1}|& \leq |\nabla(\mathcal{I}_{2s}(|\nabla u_k|^q))|+|\nabla(\mathcal{I}_{2s}(\omega))|\\&\leq C_0\left( \mathcal{I}_{2s-1}(|\nabla u_k|^q)+\mathcal{I}_{2s-1}(\omega)\right) \\& \leq C_0 \left(\mathcal{I}_{2s-1}(a_k^q\left[ \mathcal{I}_{2s-1}(\omega)\right] ^q)+\mathcal{I}_{2s-1}(\omega)\right) \quad \text{by \eqref{induc}}\\&\leq C_0\left(a_k^qC_1\mathcal{I}_{2s-1}(\omega)+\mathcal{I}_{2s-1}(\omega) \right)\quad \text{by \eqref{ineqpot}}\\& = C_0(a_k^qC_1+1)\mathcal{I}_{2s-1}(\omega).  	
	\end{split}
	\end{equation*}Then
	$$|\nabla u_{k+1}|\leq a_{k+1}\mathcal{I}_{2s-1}(\omega)$$with $a_{k+1}=C_0(a_k^qC_1+1)$.
	Then, if $C_1\leq (q')^{1-q}q^{-1}C_0^{-q}$ we see that
	$$\lim_{k\to\infty}a_k=a\leq C_0q'$$where $a$ is a root of the equation $x=C_0(x^qC_1+1)$. Hence \eqref{ineq1} holds with $C_2=C_0q'$. 
	Now we prove \eqref{ineq2}. Assume \eqref{ineqpot} holds with $C_1\leq (q')^{1-q}q^{-1}C_0^{-q}$ so that \eqref{ineq1} is satisfied with $C_2=C_0q'$, where $C_0$ is the constant of \eqref{eqgrad}. 
	
	Now,
	
	\begin{equation*}
	u_1-u_0  =  \mathcal{I}_{2s}(|\nabla u_0|^q).
	\end{equation*}Then,
	
	\begin{equation*}
	\begin{split}
	|\nabla u_1 - \nabla u_0| & \leq C_0\mathcal{I}_{2s-1}(|\nabla u_0|^q) \\& \leq C_0 C_2^q \mathcal{I}_{2s-1}(\left[ \mathcal{I}_{2s-1}(\omega)\right] ^q)\\& \leq C_0C_2^qC_1\mathcal{I}_{2s-1}(\omega).
	\end{split}
	\end{equation*}Therefore
	
	$$|\nabla u_1 - \nabla u_0|\leq b_0 \mathcal{I}_{2s-1}(\omega)$$with $b_0=C_0C_2^qC_1$. Analogously, 
	
	$$u_{k+1}-u_k = \mathcal{I}_{2s}(|\nabla u_k|^q-|\nabla u_{k-1}|^q)$$and
	
	\begin{equation}\label{eqdifgrad}
	|\nabla u_{k+1}-\nabla u_k|\leq C_0\mathcal{I}_{2s-1}(|\nabla u_k|^q-|\nabla u_{k-1}|^q).
	\end{equation}Using the inequality $|r^q-s^q|\leq q|r-s|\max\{r,s\}^{q-1}$ with $r=|\nabla u_k|$ and $s=|\nabla u_{k-1}|$ we obtain
	
	\begin{equation}\label{eqdifgrad2}
	\begin{split}
	\left| |\nabla u_k|^q-|\nabla u_{k-1}|^q\right|&\leq q\left| |\nabla u_k|-|\nabla u_{k-1}|\right|\max\{|\nabla u_k|,|\nabla u_{k-1}|\}^{q-1}\\& \leq C_2^{q-1}q|\nabla u_k-\nabla u_{k-1}|\left[ \mathcal{I}_{2s-1}(\omega)\right] ^{q-1}.   
	\end{split}
	\end{equation}Thus by \eqref{eqdifgrad} and \eqref{eqdifgrad2}
	
	\begin{equation}\label{ineqgrad3}
	|\nabla u_{k+1}-\nabla u_k|\leq C_0C_2^{q-1}q\mathcal{I}_{2s-1}(|\nabla u_k-\nabla u_{k-1}|\left[ \mathcal{I}_{2s-1}(\omega)\right] ^q).
	\end{equation}Now suppose $|\nabla u_k-\nabla u_{k-1}|\leq b_k\mathcal{I}_{2s-1}(\omega)$. Then, by \eqref{ineqgrad3}
	
	\begin{equation*}
	\begin{split}
	|\nabla u_{k+1}-\nabla u_k|&\leq C_0C_2^{q-1}q\mathcal{I}_{2s-1}(b_k\left[ \mathcal{I}_{2s-1}(\omega)\right] ^q)\\& \leq C_0C_2^{q-1}qb_kC_1\mathcal{I}_{2s-1}(\omega).
	\end{split}
	\end{equation*}Then, arguing by induction we see that
	
	$$|\nabla u_{k+1}-\nabla u_k|\leq b_{k+1}\mathcal{I}_{2s-1}(\omega)$$with $b_{k+1}\leq C_0C_2^{q-1}qC_1b_k$. Thus,
	$$b_{k+1}\leq \left( C_0C_2^{q-1}qC_1\right)^{k+1}b_0$$with $b_0=C_0C_2^qC_1$.
	Taking $C_1$ such that $\delta=C_0C_2^{q-1}qC_1<1$, we obtain \eqref{ineq2} with $C_3=b_0$.

	Now we claim that
	
	\begin{equation}\label{ineq3}
	|u_{k+1}-u_k|\leq C_4\delta^k\mathcal{I}_{2s}(\left[ \mathcal{I}_{2s-1}(\omega)\right] ^q)
	\end{equation}with $C_4>0$ and $0<\delta<1$ depending only on $q$, $n$ and $s$.
	Indeed, by \eqref{eqdifgrad2} and \eqref{ineq2}
	
	\begin{equation*}
	\begin{split}
	|u_{k+1}-u_k|& \leq \mathcal{I}_{2s}\left(\left| |\nabla u_k|^q-|\nabla u_{k-1}|^q\right|  \right)\\&\leq C_2^{q-1}q\mathcal{I}_{2s}(|\nabla u_k-\nabla u_{k-1}|\left[ \mathcal{I}_{2s-1}(\omega)\right] ^{q-1})\\&\leq C_2^{q-1}qC_3\delta^k\mathcal{I}_{2s}(\left[ \mathcal{I}_{2s-1}(\omega)\right] ^q).  
	\end{split}
	\end{equation*}Then \eqref{ineq3} holds for $C_4=C_2^{q-1}qC_3$.
	
	Now suppose \eqref{ineqpot} holds for $C_1$ small enough so that \eqref{ineq1} and \eqref{ineq2} are satisfied. Let
	
	$$u(x)=u_0(x)+\sum_{k=0}^{\infty}(u_{k+1}(x)-u_k(x))$$with $u_k$ defined by \eqref{seq}. By \eqref{ineq3}, 
	$$|u_{k+1}(x)-u_k(x)|\leq C_4\delta^k\mathcal{I}_{2s}(\left[ \mathcal{I}_{2s-1}(\omega)\right]^q).$$Hence, $u(x)=\lim_{k\to\infty}u_k(x)$ and $u\in L_s(\mathbb{R}^n)$ since 
	
	\begin{equation}\label{estimate 2}
	|u|\leq C\left( \mathcal{I}_{2s}\omega + \mathcal{I}_{2s}\left[ \mathcal{I}_{2s-1}\omega\right]^q \right)
	\end{equation}On the other side, by \eqref{ineq2} 
	$$|\nabla u_{k+1}-\nabla u_k|\leq C_3\delta^k\mathcal{I}_{2s-1}(\omega),$$hence
\begin{equation}\label{to represent}
|\nabla u|\leq C\mathcal{I}_{2s-1}(\omega).
\end{equation}	Then, by Lemma \ref{regriesz} $|\nabla u|\in L^q(\mathbb{R}^n)$. Therefore, $u\in W_{\textnormal{loc}}^{1,q}(\mathbb{R}^n)\cap L_s(\mathbb{R}^n)$ (observe that $u \in L_{\text{loc}}^{q}(\mathbb{R}^{n})$ by \eqref{for the pf}).
	
	

	
	
	We now check that Definition \ref{Defsol} holds. Let $\varphi\in S(\mathbb{R}^n)$.  We have that $\nabla u(x)=\lim_{k\to\infty}\nabla u_k(x)=\nabla u_0(x)+\sum_{k=0}^{\infty}(\nabla u_{k+1}(x)-\nabla u_k(x))$ a.e. in $\mathbb{R}^n$. Also, $|\nabla u_k|^q\leq C\left[ \mathcal{I}_{2s-1}(\omega)\right] ^q\in L^1(\mathbb{R}^n)$ for all $k$. Then, letting $k\to\infty$
	$$\int_{\mathbb{R}^n}\varphi |\nabla u_k|^q\;dx \rightarrow\int_{\mathbb{R}^n}\varphi |\nabla u|^q\;dx.$$
	
	
	
	On the other side observe that, for all k
	$$|u_k|\leq C\left[  \mathcal{I}_{2s}(\omega) + \mathcal{I}_{2s}\left( \left[\mathcal{I}_{2s-1}(\omega) \right]^q \right) \right].$$Thus, since $\varphi\in\mathcal{S}(\mathbb{R}^n)$ and $\mathcal{I}_{2s}(\omega)$, $\mathcal{I}_{2s}\left( \left[ \mathcal{I}_{2s-1}(\omega)\right]^q\right) \in L_s(\mathbb{R}^n)$, we can apply Lebesgue dominated convergence theorem and obtain
	
	$$\int_{\mathbb{R}^n}u_{k+1}(-\Delta)^s\varphi\;dx \to \int_{\mathbb{R}^n}u(-\Delta)^s\varphi\;dx$$for $k\to\infty$.

	Now, since $(-\Delta)^su_{k+1}=|\nabla u_k|^q + \omega$ in $\mathbb{R}^n$, by Definition \ref{Defsol}
	
	$$\int_{\mathbb{R}^n}u_{k+1}(-\Delta)^s\varphi\;dx = \int_{\mathbb{R}^n}\varphi |\nabla u_k|^q\;dx +\int_{\mathbb{R}^n}\varphi\;d\omega.$$
	
	Consequently, letting $k\to \infty$ in the preceding equality we get the desired conclusion.

	Finally, we prove the representation  \eqref{represent}. By \eqref{seq} is enough to show that, as $k\to\infty$, $\mathcal{I}_{2s}(|\nabla u_k|^q)\to \mathcal{I}_{2s}(|\nabla u|^q)$.

First, by definition of Riesz potential we have

\begin{equation*}
|\mathcal{I}_{2s}(|\nabla u_k|^q)-\mathcal{I}_{2s}(|\nabla u|^q)|  \leq c(n,2s)\int_{\mathbb{R}^n}\frac{||\nabla u_k|^q-|\nabla u|^q|}{|x-y|^{n-2s}}\,dy
\end{equation*}

Now, by \eqref{ineq1} and \eqref{to represent} we get

\begin{equation*}
\frac{||\nabla u_k|^q-|\nabla u|^q|}{|x-y|^{n-2s}}\leq \frac{q\left|\nabla u_k-\nabla u \right|(\max\{|\nabla u_k|,|\nabla u|\})^{q-1}}{|x-y|^{n-2s}}\leq \frac{qC\left[ \mathcal{I}_{2s-1}(\omega)\right] ^q}{|x-y|^{n-2s}}
\end{equation*}

In Proposition 2.12, we proved 
\begin{equation*}
\mathcal{I}_{2s}(\left[ \mathcal{I}_{2s-1}(\omega)\right] ^q)(x)\leq C(q,n,s)\left[ \mathcal{I}_{2s-1}(\omega)(x) + \|\mathcal{I}_{2s-1}(\omega)\|^q_{L^q(\mathbb{R}^n)}\right]
\end{equation*}in $\mathbb{R}^n$. Therefore, 
$$\int_{\mathbb{R}^{n}}\frac{qC\left[ \mathcal{I}_{2s-1}(\omega)\right] ^q}{|x-y|^{n-2s}}dy= C(n,q,s)\mathcal{I}_{2s}(\left[ \mathcal{I}_{2s-1}(\omega)\right] ^q)(x)$$ is finite, so we can apply Lebesgue's Theorem to obtain

\begin{equation*}
\lim_{k\to\infty}|\mathcal{I}_{2s}(|\nabla u_k|^q)-\mathcal{I}_{2s}(|\nabla u|^q)| = c(n,2s)\int_{\mathbb{R}^n}\lim_{k\to\infty}\frac{||\nabla u_k|^q-|\nabla u|^q|}{|x-y|^{n-2s}}\,dy = 0
\end{equation*}

Hence, we get \eqref{represent}.

\end{proof}

\section{Proof of Theorem \ref{partial converse}}\label{converse sec}

  Suppose that  \eqref{eq} has a solution $u \in W^{1, q}(\mathbb{R}^{n})$.  Let $u_k \in \mathcal{C}^{\infty}_0(\mathbb{R}^{n})$ be a sequence converging to $\widetilde{u}$ in $W^{1, q}(\mathbb{R}^{n})$. Then for non-negative $\varphi \in \mathcal{C}^{\infty}_0(\mathbb{R}^{n})$ with $\varphi \geq \chi_E$ we have
\begin{equation}
\begin{split}
\omega(E)& \leq \int_{\mathbb{R}^{n}}\varphi d\omega \\ &  \leq \int_{\mathbb{R}^{n}}u(-\Delta)^{s}\varphi\quad \text{since $u$ is a solution}\\ & = \int_{\mathbb{R}^{n}}u_k(-\Delta)^{s}\varphi + o(1) \\ & = \int_{\mathbb{R}^{n}}(-\Delta)^{1/2}u_k(-\Delta)^{(2s-1)/2}\varphi+ o(1)  \quad \text{ by Lemma 2.2 in \cite{CV1}}.
\end{split}
\end{equation}By continuity of the operator $(-\Delta)^{1/2}: W^{1, q}(\mathbb{R}^{n}) \to L^{q}(\mathbb{R}^{n})$ (see for instance \cite[Theorem 2.1]{DK}) and H\"{o}lder's inequality,  we derive
\begin{equation}\label{converse 1}
\begin{split}
\omega(E) & \leq \int_{\mathbb{R}^{n}}(-\Delta)^{1/2}u(-\Delta)^{(2s-1)/2}\varphi \\ & \leq \|(-\Delta)^{1/2}u\|_{L^{q}(\mathbb{R}^{n})}  \|(-\Delta)^{(2s-1)/2}\varphi|\|_{L^{p}(\mathbb{R}^{n})}     \\ & \leq C\|\nabla u\|_{L^{q}(\mathbb{R}^{n})} \| \varphi\|_{\mathcal{L}^{p}_{2s-1}(\mathbb{R}^{n})}, 
\end{split}
\end{equation}where for $r \in (1, \infty)$ and $\alpha >0$, $\mathcal{L}_{\alpha}^{r}(\mathbb{R}^{n})$ denotes the  Bessel potential space defined as
$$\mathcal{L}_{\alpha}^{r}(\mathbb{R}^{n}):= \left\lbrace v \in L^{r}(\mathbb{R}^{n}): (I-\Delta)^{\alpha/2}v \in L^{r}(\mathbb{R}^{n}) \right\rbrace$$with the norm
 $$\| v\|_{\mathcal{L}_{\alpha}^{r}(\mathbb{R}^{n})}= \|v\|_{L^{r}(\mathbb{R}^{n})}+ \|(-\Delta)^{\alpha/2}v\|_{L^{r}(\mathbb{R}^{n})}.$$Since $\omega$ has compact support and the Riesz capacity of $E$ is zero, we obtain from \cite[Proposition 5.1.4 (b)]{AH} that the Bessel capacity of $E$, denoted by  $ C_{2s-1, p}(E)$, is also zero. Then  there is a sequence of non-negative and smooth functions $\varphi_k \in \mathcal{L}_{2s-1}^{p}(\mathbb{R}^{n})$ such that
 $$\varphi_k \geq \chi_E, \quad \varphi_k \to 0 \quad \text{in }\mathcal{L}_{2s-1}^{p}(\mathbb{R}^{n}).$$Applying \eqref{converse 1} to the sequence $\varphi_k$ and taking $k \to \infty$, we conclude
 $$\omega(E) =0.$$This ends the proof.

 \begin{remark} We now comment on the  necessity of \eqref{cap1} for existence of solutions. We have not proved the converse of Theorem \ref{Theo1}, but we exhibit below how some arguments may be applied to find a relation between the measure $\omega$, the Riesz capacity and the solution $u$.  We consider $\varphi$ as before and we take $\varphi^{p}$ as a test function to derive
\begin{equation}
\begin{split}
\omega(E)& \leq  \int_{\mathbb{R}^{n}}u_k(-\Delta)^{s}\varphi^{p}-\int_{\mathbb{R}^{n}}|\nabla u|^{q}\varphi^{p} + o(1) \\ & = \int_{\mathbb{R}^{n}}(-\Delta)^{1/2}u_k(-\Delta)^{(2s-1)/2}\varphi^{p}-\int_{\mathbb{R}^{n}}|\nabla u|^{q}\varphi^{p} + o(1)  
\end{split}
\end{equation}Letting $k \to \infty$ and by Young's inequality, we get
\begin{equation}\label{converse 11}
\begin{split}
\omega(E) & \leq \int_{\mathbb{R}^{n}}(-\Delta)^{1/2}u(-\Delta)^{(2s-1)/2}\varphi^{p}-\int_{\mathbb{R}^{n}}|\nabla u|^{q}\varphi^{p} \\ &  \leq C_p\int_{\mathbb{R}^{n}}|(-\Delta)^{1/2}u| \varphi^{p-1}(-\Delta)^{(2s-1)/2}\varphi-\int_{\mathbb{R}^{n}}|\nabla u|^{q}\varphi^{p}  \\ &  \leq  \int_{\mathbb{R}^{n}} |(-\Delta)^{1/2}u|^{q}\varphi^{p} + C_p \int_{\mathbb{R}^{n}}|(-\Delta)^{(2s-1)/2}\varphi|^{p} -\int_{\mathbb{R}^{n}}|\nabla u|^{q}\varphi^{p}\\ & = \int_{\mathbb{R}^{n}}\left( |(-\Delta)^{1/2}u|^{q}- |\nabla u|^{q}\right)\varphi^{p}  + C_p \| \varphi\|^{p}_{\mathcal{L}_{2s-1}^{p}(\mathbb{R}^{n})}.
\end{split} \end{equation}
 Taking the \'infimum over $\varphi$  in \eqref{converse 11}, we derive
\begin{equation*}
\omega(E) \leq \inf_{\varphi}\int_{\mathbb{R}^{n}}\left(|(-\Delta)^{1/2}u|^{q}- |\nabla u|^{q}\right)\varphi^{p}  + C_p C_{2s-1, p}(E).
\end{equation*}Since $\omega$ has compact support, we derive from \cite[Proposition 5.1.4 (b)]{AH} that there is a constant $C> 0$ so that
\begin{equation}\label{est with inf}
\omega(E) \leq \inf_{\varphi}\int_{\mathbb{R}^{n}}\left(|(-\Delta)^{1/2}u|^{q}- |\nabla u|^{q}\right)\varphi^{p}  + C \text{cap}_{2s-1, p}(E).
\end{equation}In the local case, exposed in \cite{Han},  the first term in \eqref{est with inf} does not appear and hence \eqref{cap1} is also necessary for existence. In the current scenario, we have not obtained the vanishing of that term. Indeed, adapting directly the proof of \cite[Lemma 2.1]{Han} to our case does not seem to be straightforward and the problem is left as open.
 \end{remark}


\begin{thebibliography}{00}
	
	
	\bibitem{AyP} Abdellaoui B., Peral I.: Towards a deterministic KPZ equation with fractional diffusion: the stationary problem.  Nonlinearity 31, 1260-1298 (2018).
	
	\bibitem{AH} Adams D. R.,  Hedberg L.: Function spaces and Potential theory.  Springer-Verlag, Berlin (1999). 
	
	\bibitem{AP} Adams D. R.,  Pierre M.:  Capacitary strong type estimates in semilinear problems.  Ann. Inst. Fourier  41 (1), 117-135  (1991). 
	\bibitem{A04} Applebaum D.: L\'{e}vy processes - From probability to finance and quantum groups. Notices Amer. Math. Soc. 51  (11), 1336-1347 (2004).
	\bibitem{BCF} Bjorland C. , Caffarelli L., Figalli A.: Non-local gradient dependent operators. Adv. Math. 230,  1859-1894 (2012).
	\bibitem{BCF12} Bjorland C. , Caffarelli L.,  Figalli A.:   Nonlocal tug-of-war and the infinity fractional Laplacian. Comm. Pure Appl. Math. 6  (3), 337-380 (2012).
		\bibitem{BGO2} Boccardo L., Gallou\"{e}t T.,  Orsina L.: Existence and nonexistence of  solutions for some nonlinear  elliptic equations, J. Anal. Math. 73, 203-223 (1997). 
	\bibitem{BGO} Boccardo L., Gallou\"{e}t T.,  Orsina L.: Existence and uniqueness of entropy solutions for nonlinear equations with measure data, Ann. Inst. Henri Poincar\'e 13  (5), 539-551  (1996). 
	\bibitem{Buc} Bucur C.: Some observations on the Green function for the ball in the fractional Laplace framework.   Communications on Pure and Applied Analysis 15, 657 - 699 (2016). 
	\bibitem{BV} Bucur C.,   Valdinocci E.: Nonlocal diffusion and applications. Lecture Notes of the Unione Matematica Italiana, Springer (2016).
	\bibitem{C} Caffarelli L.: Non local operators, drifts and games. Nonlinear PDEs. Abel Symposia 7, 37-52 (2012).
	\bibitem{CV10} Caffarelli L., Vasseur A.: Drift diffusion equations with fractional diffusion and the quasi-geostrophic equation. Ann. Math. 171, 1903-1930, (2010) .
	\bibitem{CH} Chen, H. Alhomedan S., Hajaiej H.,  Markowich P.: Complete study of the existence and uniqueness of solutions for semilinear elliptic equations involving measures concentrated on boundary. Complex Var. Elliptic Equ. 62 (12), 1687-1729 (2017).
	\bibitem{CV1} Chen H.,  V\'{e}ron L.: Semilinear fractional elliptic equations involving measures. J. Differential Equations 257 (5), 1457-1486 (2014).

\bibitem{CV} Chen H.,  V\'{e}ron L.: Semilinear fractional elliptic equations with gradient nonlinearity involving measures.  J. Funct. Anal. 266 (8), 5467-5492 (2014).
\bibitem{CT} Cont R.,  Tankov P.: Financial modelling with jump processes. Chapman $\&$ Hall/CRC Financial Mathematics Series. Chapman $\&$ Hall/CRC, Boca Raton, FL, 2004. xvi+535 pp. ISBN: 1-5848-8413-4.
\bibitem{DFV} Dipierro S., Figalli A.,  Valdinoci E.: Strongly non local dislocation dynamics in crystals. Comm. Partial Differential Equations  39 (12),  2351-2387 (2014).
\bibitem{DPV} Dipierro S., Palatucci G.,  Valdinoci E.: Dislocation dynamics in crystals: a macroscopic theory in a fractional Laplace setting. Comm. Math. Phys., 333 (2), 1061-1105 (2011).
\bibitem{GO} Gilboa G., Osher S.: Non-local operators with applications to image processing. Miltiscale Model. Simul. 7, 1005-1028 (2008).
	\bibitem{Han} Hansson K., Maz'ya V., Verbitsky E.: Criteria of solvability for multidimensional Riccati equation.  Institut Mittag-Leffler. Ark. Mat. 37, 87-120 (1999). 
	\bibitem{KPZ} Kardar M., Parisi G.,  Zhang Y. C.: Dynamic scaling of growing interfaces. Phys. Rev. Lett. 56, 889-892 (1986).
	\bibitem{KV} Kalton N. J.,  Verbitsky I. E.: Nonlinear equations and weighted norm inequalities, Transaction of the American Mathematical Society 351 (9), 3441-3497 (1999).
	\bibitem{Land}Landkof N. S.: Foundations of Modern Potential Theory.  Springer-Verlag, Berlin, Heidelberg, New York  (1972).
	\bibitem{L00} Laskin N.: Fractional quantum mechanics and L\'{e}vy path integrals. Physics Letters A  268 (4),  298-305 (2000).
	
	\bibitem{LP} Leonori T., Peral I., Primo A.,  Soria F.: Basic estimates for solutions of a class of nonlocal elliptic and parabolic equations. Discrete and Continuous Dynamical Systems 35 12, 6031-6068 (2015).
	\bibitem{Mar}Martinazzi L.: Fractional Adams-Moser-Trudinger type inequalities.  University of Basel, Department of Mathematics and Computer Science, Switzerland, Nonlinear Analysis 127, 263-278  (2015).
	
	
	
	\bibitem{Ma} Maz'ya V.: Sobolev spaces.  Springer-Verlag, Berlin, (2011).
	
	\bibitem{MV} Maz'ya V.,  I. Verbitsky, Capacitary inequalities for fractional integrals, with applications to partial differential equations and Sobolev multipliers, Ark Mat, 33, 81-115 (1995).  
	\bibitem{MK} Metzler R.,  Klafter J.: The random walk's guide to anomalous diffusion: a fractional dynamics approach. Phys. Rep. 339, 1-77 (2000).
%
\bibitem{MK1} Metzler R., Klafter J.: The restaurant at the random walk: recent developments in the description of anomalous transport by fractional dynamics. J. Phys. A 37, 161-208 (2004).
\bibitem{Mi} Mizuta Y.: Potential theory in Euclidean spaces. Gakuto International Series, (1996). 
\bibitem{Ponce} Ponce, A.: Elliptic PDEs, Measures and Capacities, European Mathematical Society, (2016).
	\bibitem{S} Silvestre L.: Regularity of the obstacle problem for a fractional power of the Laplace operator.  Comm. Pure Appl. Math. 60, 67–112 (2007).
	
	\bibitem{St} Stein, E. M.: Singular integrals and differentiability properties of functions.  Princeton, NJ: Princeton University Press, (1970). 
	

	
	
\end{thebibliography}
\end{document}